\documentclass[11pt,twoside]{article}
\usepackage{mathrsfs}
\usepackage{amsmath}
\usepackage{amsthm}
\input amssymb.sty
\usepackage{amsfonts}
\usepackage{amssymb}
\usepackage{latexsym}
\usepackage[all]{xy}

\date{\empty}
\allowdisplaybreaks[4] \footskip=15pt

\numberwithin{equation}{section} \theoremstyle{plain}
\newtheorem*{thm*}{Main Theorem}
\newtheorem{theorem}{Theorem}[section]

\newtheorem*{corollary*}{Corollary}

\newtheorem*{claim*}{Claim}
\newtheorem{lemma}[theorem]{Lemma}
\newtheorem*{lemma*}{Lemma}

\newtheorem*{proposition*}{Proposition}

\newtheorem*{remark*}{Remark}

\newtheorem*{example*}{Example}

\newtheorem*{question*}{Question}

\newtheorem*{definition*}{Definition}

\begin{document}
\begin{center}
{\large  \bf Additive Property of Drazin Invertibility of Elements}\\
\vspace{0.8cm} {\small {\bf Long Wang,  \ \ Huihui Zhu, \ \ Xia Zhu, \ \ Jianlong Chen\footnote{Corresponding author. Email: jlchen@seu.edu.cn}} \\
Department of Mathematics, Southeast University, Nanjing 210096, China.} 
\end{center}

\bigskip

{ \bf  Abstract:}  \leftskip0truemm\rightskip0truemm In this
article, we investigate additive properties of the Drazin inverse of
elements in rings and algebras over an arbitrary field. Under the weakly
commutative condition of $ab = \lambda ba$, we show that $a-b$ is
Drazin invertible if and only if $aa^{D}(a-b)bb^{D}$ is Drazin
invertible. Next, we give explicit representations of
$(a+b)^{D}$, as a function of $a, b, a^{D}$ and $b^{D}$, under the
conditions $a^{3}b = ba$ and $b^{3}a = ab$.
 \bigskip

%%%%%%%%%%%%%%%%%%%%%%%%%%%%%%%%%%%%%%%%%%%%%%%%%%%%%%%%%%%%%%%%%%%%%
%%%%%%%%%%%%%%%%%%%%%    Section 1   %%%%%%%%%%%%%%%%%%%%%%%%%%%%
%%%%%%%%%%%%%%%%%%%%%%%%%%%%%%%%%%%%%%%%%%%%%%%%%%%%%%%%%%%%%%%%%%%%%

\section { \bf Introduction}

Throughout this article, $\mathcal{A}$ denotes an algebra over an
arbitrary field $\mathbb{F}$ and $R$ is an associative ring with
unity. Recall that the Drazin inverse of $a\in R$ is the element
$b\in R$ (denoted by $a^{D}$) which satisfies the following
equations \cite{DRA}:
$$bab=b,\ \ \ ab=ba,\ \ \ a^{k}=a^{k+1}b.$$
for some nonnegative integer $k$. The smallest integer $k$ is called
the Drazin index of $a$, denoted by $\textrm{ind}(a)$. If
$\textrm{ind}(a)=1$, then $a$ is group invertible and the group
inverse of $a$ is denoted by $a^{\sharp}$. It is well known that the
Drazin inverse is unique, if it exists. The conditions in the
definition of Drazin inverse are equivalent to:
$$bab=b,\ \ \ ab=ba,\ \ \ a-a^{2}b \ \ \textrm{is} \ \ \textrm{nilpotent}.$$
The study of the Drazin inverse of the sum of two Drazin invertible elements
was first developed by Drazin \cite{DRA}. It was proved that $(a+b)^{D}=a^{D}+b^{D}$ provided that $ab=ba=0$. In recent years, many
papers focus on the problem under some weaker conditions.
Hartwig et al. for matrices \cite{HARTWING2}
expressed $(a+b)^{D}$ under the one-side condition $ab=0$.
This result was extended to bounded linear operators on an arbitrary
complex Banach space by Djordjevi$\acute{c}$ and Wei \cite{Dj1}, and was extended for morphisms
on arbitrary additive categories by Chen et al. \cite{C}.
In the article of Wei and
Deng \cite{W1} and Zhuang et al. \cite{Z}, the commutativity $ab=ba$ was assumed. In \cite{W1}, they
characterized the relationships of the Drazin inverse between $A+B$ and $I+A^{D}B$ by
Jordan canonical decomposition for complex matrices $A$ and $B$.
In \cite{Z}, Zhuang et al. extended the result in \cite{W1} to a ring $R$ , and it was shown that if $a, b \in R$ are Drazin invertible and
$ab=ba$, then $a+b$ is Drazin invertible if and only if $1+a^{D}b$ is Drazin invertible.
More results on the Drazin inverse
 can also be found in [1-3, 6, 7, 9, 11, 13, 14, 16, 17, 19-24].
The motivation for this article was the article of Deng \cite{D1}, Cvetkovi$\acute{c}$-Ili$\acute{c}$ \cite{DS1} and Liu et al. \cite{L1}. In \cite{DS1,D1} the commutativity $ab=\lambda ba$ was assumed. In \cite{D1}, the author
characterized the relationships of the Drazin inverse between $a\pm b$ and $aa^{D}(a\pm b)bb^{D}$ by
the space decomposition for operator matrices $a$ and $b$.
In \cite{L1}, the author gave explicit representations of $(a+b)^{D}$ of two
matrices $a$ and $b$, as a function of $a,b,a^{D}$ and $b^{D}$, under the conditions $a^{3}b = ba$ and $b^{3}a = ab$.
In this article, we extend the results in \cite{D1,L1} to more general settings.

As usual, the set of all Drazin invertible elements in an algebra
$\mathcal{A}$ is denoted by $\mathcal{A}^D$. Similarly, $R^D$
indicates the set of all Drazin invertible elements in a ring $R$.
Given $a\in \mathcal{A}^D$ (or $a\in R^D$), it is easy to see that
$1-aa^D$ is an idempotent, which is denoted by $a^{\pi}$.

\section { \bf Under the condition $ab=\lambda ba$}

 In this section, we will extend the result in \cite{D1} to an algebra $\mathcal{A}$ over an arbitrary field $\mathbb{F}$.
\begin{lemma}
Let $a, b\in \mathcal{A}$ be such that $ab=\lambda ba$ and
$\lambda\in \mathbb{F}\backslash\{0\}$. Then

$(1)$ $ab^{i}=\lambda^{i}b^{i}a$\ \ and\ \ $a^{i}b=\lambda^{i}ba^{i}$.

$(2)$ $(ab)^{i}=\lambda^{-\sum^{k=i-1}_{k=1} k}a^{i}b^{i}$\ \ and \
\ $(ba)^{i}=\lambda^{\sum^{k=i-1}_{k=1}k}b^{i}a^{i}$.
\end{lemma}
\begin{proof}
(1) By hypothesis, we have
$$ab^{i}=abb^{i-1}=\lambda
bab^{i-1}=\lambda babb^{i-2}=\lambda^{2}
b^{2}ab^{i-2}=\ldots=\lambda^{i}b^{i}a.$$
Similarly, we can obtain
that $a^{i}b=\lambda^{i}ba^{i}.$

(2) By hypothesis, it follows that
\begin{eqnarray*}
(ab)^{i}&=&abab(ab)^{i-2}=\lambda^{-1}a^{2}b^{2}(ab)^{i-2}=\lambda^{-(1+2)}a^{3}b^{3}(ab)^{i-3}\\
&=&\ldots=\lambda^{-\sum^{k=i-1}_{k=0} k}a^{i}b^{i}.
\end{eqnarray*}
Similarly, it is easy to get $(ba)^{i}=\lambda^{\sum^{k=i-1}_{k=0}
k}b^{i}a^{i}.$
\end{proof}
\begin{lemma}
Let $a, b\in \mathcal{A}$ be Drazin invertible and $\lambda\in
\mathbb{F}\backslash\{0\}$. If $ab=\lambda ba$, then

$(1) \ a^{D}b=\lambda^{-1} ba^{D}.$

$(2) \ ab^{D}=\lambda^{-1} b^{D}a.$

$(3) \ (ab)^{D}=b^{D}a^{D}=\lambda^{-1} a^{D}b^{D}.$
\end{lemma}
\begin{proof}
Assume $k=\max\{\textrm{ind}(a), \ \textrm{ind}(b)\}.$ \\
(1) By hypothesis, we have
\begin{eqnarray*}
a^{D}(a^{k}b)&=&a^{D}(\lambda^{k}ba^{k})
=\lambda^{k}a^{D}(ba^{k+1}a^{D})
=\lambda^{k}a^{D}(\lambda^{-(k+1)}a^{k+1}ba^{D})\\
&=&\lambda^{-1}a^{D}a^{k+1}ba^{D}
=\lambda^{-1}a^{k}ba^{D}.
\end{eqnarray*}
Then it follows that
\begin{eqnarray*}
a^{D}b&=&(a^{D})^{k+1}a^{k}b=(a^{D})^{k}a^{D}a^{k}b=\lambda^{-1}(a^{D})^{k}a^{k}ba^{D}=\ldots\\
&=&\lambda^{-(k+1)}a^{k}b(a^{D})^{k+1}=\lambda^{-1}ba^{k}(a^{D})^{k+1}
=\lambda^{-1} ba^{D}.
\end{eqnarray*}

(2) The proof is similar to (1).

(3) By (1), we have $a^{D}b=\lambda^{-1} ba^{D}$, then $(aa^{D})b=\lambda^{-1} aba^{D}=b(aa^{D})$. By \cite{DRA}, we get
$aa^{D}b^{D}=b^{D}aa^{D}$. \\
Similarly, we can obtain that $ab^{D}b=\lambda^{-1} b^{D}ab=b^{D}ba$ and $a^{D}bb^{D}=bb^{D}a^{D}$.
This implies that
 $$abb^{D}a^{D}=bb^{D}aa^{D}=b^{D}a^{D}ab.$$
 $$b^{D}a^{D}abb^{D}a^{D}=b^{D}bb^{D}a^{D}aa^{D}=b^{D}a^{D}.$$
and
\begin{eqnarray*}
(ab)^{k+1}b^{D}a^{D}&=&\lambda^{-\sum^{i=k}_{i=0} i}a^{k+1}b^{k+1}b^{D}a^{D}=\lambda^{-\sum^{i=k}_{i=0} i}a^{k+1}b^{k}a^{D}\\
 &=&\lambda^{-\sum^{i=k}_{i=0} i}a^{k+1}(\lambda^{k}a^{D}b^{k})=\lambda^{-\sum^{i=k-1}_{i=0} i}a^{k+1}a^{D}b^{k}\\
 &=&\lambda^{-\sum^{i=k-1}_{i=0} i}a^{k}b^{k}
 =(ab)^{k}.
\end{eqnarray*}
Then we get $(ab)^{D}=b^{D}a^{D}$. Similarly, we can check that
$(ab)^{D}=\lambda^{-1} a^{D}b^{D}.$
\end{proof}

\begin{theorem}
Let $a, b$ be Drazin invertible in $\mathcal{A}$. If $ab=\lambda ba$ and
$\lambda \neq0,$\ then $a-b$ is Drazin invertible if and only if
$w=aa^{D}(a-b)bb^{D}$ is Drazin invertible, in this case we
have$$(a-b)^{D}=w^{D}+a^{D}(1-bb^{\pi}a^{D})^{-1}b^{\pi}-a^{\pi}(1-b^{D}aa^{\pi})^{-1}b^{D}.$$
\end{theorem}
\begin{proof}
Since $w=aa^{D}(a-b)bb^{D}$, we have $w=(1-a^{\pi})(a-b)(1-b^{\pi})$ and
$$a-b=w+(a-b)b^{\pi}+a^{\pi}(a-b)-a^{\pi}(a-b)b^{\pi}.$$
By the proof of Lemma 2.2 (3), we have $aa^{D}b=baa^{D}$ and
$abb^{D}=b^{D}ba$. Then
\\$a^{\pi}b=(1-aa^{D})b=b(1-aa^{D})=ba^{\pi}$ and
$b^{\pi}a=(1-bb^{D})a=a(1-bb^{D})=ab^{\pi}$.\\
Let $s=\mathrm{ind}(a)$ and $t=\mathrm{ind}(b)$.
By Lemma 2.2, we get
\begin{eqnarray*}
&&(1-bb^{\pi}a^{D})(1+bb^{\pi}a^{D}+(bb^{\pi}a^{D})^{2}+\ldots+(bb^{\pi}a^{D})^{t-1})\\
\nonumber &=&1-(bb^{\pi}a^{D})^{t}=1-(ba^{D})^{t}b^{\pi}
=1-\lambda^{-\sum^{i=t-1}_{i=0}i}b^{t}(a^{D})^{t}b^{\pi}\\
&=&1-\lambda^{-\sum^{i=t-1}_{i=0}i}b^{t+1}b^{D}b^{\pi}(a^{D})^{t}=1.
\end{eqnarray*}
Similarly, we obtain
$(1+bb^{\pi}a^{D}+(bb^{\pi}a^{D})^{2}+\ldots+(bb^{\pi}a^{D})^{t-1})(1-bb^{\pi}a^{D})=1.$\\
Hence, this implies that
\begin{eqnarray*}
(1-bb^{\pi}a^{D})^{-1}&=&1+bb^{\pi}a^{D}+(bb^{\pi}a^{D})^{2}+\ldots+(bb^{\pi}a^{D})^{t-1}\\
&=&1+ba^{D}b^{\pi}+(ba^{D})^{2}b^{\pi}+\ldots+(ba^{D})^{t-1}b^{\pi}.
\end{eqnarray*}

By a similar method, we get
\begin{eqnarray*}
(1-b^{D}aa^{\pi})^{-1}&=&1+b^{D}aa^{\pi}+(b^{D}aa^{\pi})^{2}+\ldots+(b^{D}aa^{\pi})^{s-1}\\
&=&1+b^{D}aa^{\pi}+(b^{D}a)^{2}a^{\pi}+\ldots+(b^{D}a)^{s-1}a^{\pi}.
\end{eqnarray*}

Note that $a^{\pi}w=a^{\pi}aa^{D}(a-b)bb^{D}=0$ and $wb^{\pi}=aa^{D}(a-b)bb^{D}b^{\pi}=0$.
And similarly, we have $wa^{\pi}=0$ and $b^{\pi}w=0$.

Now let us begin the proof of Theorem 2.3. Assume $w$ be Drazin invertible and let
$$x=w^{D}+a^{D}(1-bb^{\pi}a^{D})^{-1}b^{\pi}-a^{\pi}(1-b^{D}aa^{\pi})^{-1}b^{D}.$$
Since $abb^{D}=bb^{D}a$ and $baa^{D}=aa^{D}b$, we have
\begin{eqnarray*}
w(a-b)&=&aa^{D}(a-b)bb^{D}(a-b)\\
&=&aa^{D}(a-b)(a-b)bb^{D}\\
&=&(a-b)aa^{D}(a-b)bb^{D}\\
&=&(a-b)w.
\end{eqnarray*}
Then we can obtain $w^{D}(a-b)=(a-b)w^{D}$.

By directly computing, we get
\begin{eqnarray*}
&&(a-b)[a^{D}(1-bb^{\pi}a^{D})^{-1}b^{\pi}]\\
&=&(aa^{D}-ba^{D})(1+ba^{D}b^{\pi}+(ba^{D})^{2}b^{\pi}+\ldots+(ba^{D})^{t-1}b^{\pi})b^{\pi}\\
&=&[aa^{D}+aa^{D}ba^{D}b^{\pi}+aa^{D}(ba^{D})^{2}b^{\pi}+\ldots+aa^{D}(ba^{D})^{t-1}b^{\pi}]b^{\pi}\\
&&-[ba^{D}+ba^{D}ba^{D}b^{\pi}+ba^{D}(ba^{D})^{2}b^{\pi}+\ldots+ba^{D}(ba^{D})^{t-1}b^{\pi}]b^{\pi}\\
&=&[aa^{D}+ba^{D}b^{\pi}+(ba^{D})^{2}b^{\pi}+\ldots+(ba^{D})^{t-1}b^{\pi}]b^{\pi}\\
&&-[ba^{D}+(ba^{D})^{2}b^{\pi}+(ba^{D})^{3}b^{\pi}+\ldots+(ba^{D})^{t}b^{\pi}]b^{\pi}\\
&=&aa^{D}b^{\pi}-(ba^{D})^{t}b^{\pi}=aa^{D}b^{\pi}.
\end{eqnarray*}
Similarly, we have
\begin{eqnarray*}
& & (a-b)a^{\pi}(1-b^{D}aa^{\pi})^{-1}b^{D} \\
&=&(a-b)(1+b^{D}aa^{\pi}+(b^{D}a)^{2}a^{\pi}+\ldots+(b^{D}a)^{s-1}a^{\pi})b^{D}a^{\pi}\\
&=&[ab^{D}+ab^{D}ab^{D}a^{\pi}+a(b^{D}a)^{2}b^{D}a^{\pi}+\ldots+a(b^{D}a)^{s-1}b^{D}a^{\pi}]a^{\pi}\\
&&-[bb^{D}+bb^{D}ab^{D}a^{\pi}+b(b^{D}a)^{2}b^{D}a^{\pi}+\ldots+b(b^{D}a)^{s-1}b^{D}a^{\pi}]a^{\pi}\\
&=&[ab^{D}+(ab^{D})^{2}a^{\pi}+(ab^{D})^{3}a^{\pi}+\ldots+(ab^{D})^{s}a^{\pi}]a^{\pi}\\
&&-[bb^{D}+ab^{D}a^{\pi}+(ab^{D})^{2}a^{\pi}+\ldots+(ab^{D})^{s-1}a^{\pi}]a^{\pi}\\
&=&(ab^{D})^{s}a^{\pi}-bb^{D}a^{\pi}=\lambda^{\sum^{i=s-1}_{i=0} i}a^{s}(b^{D})^{s}a^{\pi}-bb^{D}a^{\pi}\\
&=&\lambda^{\sum^{i=s-1}_{i=0}i}a^{s}a^{\pi}(b^{D})^{s}-bb^{D}a^{\pi}=-bb^{D}a^{\pi}.
\end{eqnarray*}
So, by the above, we can obtain that
\begin{eqnarray*}
(a-b)x&=&(a-b)(w^{D}+a^{D}(1-bb^{\pi}a^{D})^{-1}b^{\pi}-a^{\pi}(1-b^{D}aa^{\pi})^{-1}b^{D})\\
\nonumber &=&(a-b)w^{D}+aa^{D}b^{\pi}+bb^{D}a^{\pi}.
\end{eqnarray*}
Similar to the above way, we also have
\begin{eqnarray*}
& & [a^{D}(1-bb^{\pi}a^{D})^{-1}b^{\pi}](a-b) \\
&=&a^{D}(1+ba^{D}b^{\pi}+(ba^{D})^{2}b^{\pi}+\ldots+(ba^{D})^{t-1}b^{\pi})(a-b)b^{\pi}\\
&=&(a^{D}a+a^{D}ba^{D}ab^{\pi}+a^{D}(ba^{D})^{2}ab^{\pi}+\ldots+a^{D}(ba^{D})^{t-1}ab^{\pi})b^{\pi}\\
&&-(a^{D}b+a^{D}ba^{D}bb^{\pi}+a^{D}(ba^{D})^{2}bb^{\pi}+\ldots+a^{D}(ba^{D})^{t-1}bb^{\pi})b^{\pi}\\
&=&(a^{D}a+a^{D}bb^{\pi}+(a^{D}b)^{2}b^{\pi}+\ldots+(a^{D}b)^{t-1}b^{\pi})b^{\pi}\\
&&-(a^{D}b+(a^{D}b)^{2}b^{\pi}+(a^{D}b)^{3}b^{\pi}+\ldots+(a^{D}b)^{t}b^{\pi})b^{\pi}\\
&=&a^{D}ab^{\pi}-(a^{D}b)^{t}b^{\pi}\\
&=&a^{D}ab^{\pi}-\lambda^{\sum^{i=t-1}_{i=0} i}(a^{D})^{t}(b)^{t}b^{\pi}\\
&=&a^{D}ab^{\pi}-\lambda^{\sum^{i=t-1}_{i=0} i}(a^{D})^{t}(b)^{t+1}b^{D}b^{\pi}\\
&=&a^{D}ab^{\pi}.
\end{eqnarray*}
And
\begin{eqnarray*}
& & [a^{\pi}(1-b^{D}aa^{\pi})^{-1}b^{D}](a-b) \\
&=&(1+b^{D}aa^{\pi}+(b^{D}a)^{2}a^{\pi}+\ldots+(b^{D}a)^{s-1}a^{\pi})(b^{D}a-b^{D}b)a^{\pi}\\
&=&(b^{D}a+b^{D}ab^{D}aa^{\pi}+(b^{D}a)^{2}b^{D}aa^{\pi}+\ldots+(b^{D}a)^{s-1}b^{D}aa^{\pi})a^{\pi}\\
&&-(b^{D}b+b^{D}ab^{D}ba^{\pi}+(b^{D}a)^{2}b^{D}ba^{\pi}+\ldots+(b^{D}a)^{s-1}b^{D}ba^{\pi})a^{\pi}\\
&=&(b^{D}a+(b^{D}a)^{2}a^{\pi}+(b^{D}a)^{3}a^{\pi}+\ldots+(b^{D}a)^{s}a^{\pi})a^{\pi}\\
&&-(b^{D}b+b^{D}aa^{\pi}+(b^{D}a)^{2}a^{\pi}+\ldots+(b^{D}a)^{s-1}a^{\pi})a^{\pi}\\
&=&(b^{D}a)^{s}a^{\pi}-b^{D}ba^{\pi}=\lambda^{-\sum^{i=s-1}_{i=0} i}(b^{D})^{s}a^{s}a^{\pi}-b^{D}ba^{\pi}\\
&=&-b^{D}ba^{\pi}.
\end{eqnarray*}
So, it follows
\begin{eqnarray*}
x(a-b)&=&(w^{D}+a^{D}(1-bb^{\pi}a^{D})^{-1}b^{\pi}-a^{\pi}(1-b^{D}aa^{\pi})^{-1}b^{D})(a-b)\\
&=&w^{D}(a-b)+a^{D}ab^{\pi}+b^{D}ba^{\pi}.
\end{eqnarray*}
Hence, we have $x(a-b)=(a-b)x$.

Next, we give the proof of $x(a-b)x=x$.\\
Let $(a-b)x=x_{1}+x_{2}$ where $x_{1}=w^{D}(a-b)$ and $x_{2}=a^{D}ab^{\pi}+b^{D}ba^{\pi}$.
Note that $wa^{\pi}=a^{\pi}w=0$ and $wb^{\pi}=b^{\pi}w=0$, then
\begin{eqnarray*}
 w^{D}x_{1}=(w^{D})^{2}(a-b)&=&w^{D}(w+(a-b)b^{\pi}+a^{\pi}(a-b)-a^{\pi}(a-b)b^{\pi})w^{D}\\
&=&(w^{D}w+w^{D}(a-b)b^{\pi}+w^{D}a^{\pi}(a-b)-w^{D}a^{\pi}(a-b)b^{\pi})w^{D}\\
&=&(w^{D}w+w^{D}b^{\pi}(a-b))w^{D}\\
&=&w^{D}ww^{D}=w^{D}.
\end{eqnarray*}
And $w^{D}x_{2}=w^{D}(aa^{D}b^{\pi}+bb^{D}a^{\pi})=w^{D}b^{\pi}aa^{D}+w^{D}a^{\pi}bb^{D}=0$.

Similarly, it is easy to get $(a^{D}(1-bb^{\pi}a^{D})^{-1}b^{\pi}-a^{\pi}(1-b^{D}aa^{\pi})^{-1}b^{D})w^{D}=0$, this shows that
$(a^{D}(1-bb^{\pi}a^{D})^{-1}b^{\pi}-a^{\pi}(1-b^{D}aa^{\pi})^{-1}b^{D})x_{1}=0$.

In the following, we will prove
$$[a^{D}(1-bb^{\pi}a^{D})^{-1}b^{\pi}-a^{\pi}(1-b^{D}aa^{\pi})^{-1}b^{D}]x_{2}=a^{D}(1-bb^{\pi}a^{D})^{-1}b^{\pi}-a^{\pi}(1-b^{D}aa^{\pi})^{-1}b^{D}.$$
Note that
\begin{eqnarray*}
&&(a^{D}(1-bb^{\pi}a^{D})^{-1}b^{\pi}-a^{\pi}(1-b^{D}aa^{\pi})^{-1}b^{D})(a^{D}ab^{\pi}+b^{D}ba^{\pi})\\
&=&a^{D}(1-bb^{\pi}a^{D})^{-1}b^{\pi}a^{D}ab^{\pi}+a^{D}(1-bb^{\pi}a^{D})^{-1}b^{\pi}b^{D}ba^{\pi}\\
&&-a^{\pi}(1-b^{D}aa^{\pi})^{-1}b^{D}a^{D}ab^{\pi}-a^{\pi}(1-b^{D}aa^{\pi})^{-1}b^{D}b^{D}ba^{\pi}\\
&=&a^{D}(1-bb^{\pi}a^{D})^{-1}b^{\pi}a^{D}ab^{\pi}+a^{\pi}a^{D}(1-bb^{\pi}a^{D})^{-1}b^{\pi}b^{D}b\\
&&-(1-b^{D}aa^{\pi})^{-1}b^{D}a^{\pi}a^{D}ab^{\pi}-a^{\pi}(1-b^{D}aa^{\pi})^{-1}b^{D}b^{D}ba^{\pi}\\
&=&a^{D}(1-bb^{\pi}a^{D})^{-1}b^{\pi}-a^{\pi}(1-b^{D}aa^{\pi})^{-1}b^{D}.
\end{eqnarray*}
So, we get $x(a-b)x=x$.

Since $a-b=w+(a-b)b^{\pi}+a^{\pi}(a-b)-a^{\pi}(a-b)b^{\pi}$ and $a^{\pi}w=b^{\pi}w=0$, we have
\begin{eqnarray*}
(a-b)^{2}w^{D}&=&(a-b)(w+(a-b)b^{\pi}+a^{\pi}(a-b)-a^{\pi}(a-b)b^{\pi})w^{D}\\
&=&(w+(a-b)b^{\pi}+a^{\pi}(a-b)-a^{\pi}(a-b)b^{\pi})w^{D}w\\
&=&ww^{D}w=w(1-w^{\pi})=w-ww^{\pi}
\end{eqnarray*}
and
\begin{eqnarray*}
&&(a-b)(aa^{D}b^{\pi}+bb^{D}a^{\pi})\\
&=&(a-b)((1-a^{\pi})b^{\pi}+(1-b^{\pi})a^{\pi})\\
&=&ab^{\pi}-ba^{\pi}+aa^{\pi}-bb^{\pi}-2aa^{\pi}b^{\pi}+2bb^{\pi}a^{\pi}.
\end{eqnarray*}
Then we have
\begin{eqnarray*}
&&(a-b)-(a-b)^{2}x\\
&=&(a-b)-(a-b)(w^{D}(a-b)+a^{D}ab^{\pi}+b^{D}ba^{\pi})\\
&=&(a-b)-(w-ww^{\pi}+ab^{\pi}-ba^{\pi}+aa^{\pi}-bb^{\pi}-2aa^{\pi}b^{\pi}+2bb^{\pi}a^{\pi})\\
&=&(a-b)-[(a-b)-(a-b)b^{\pi}-a^{\pi}(a-b)+a^{\pi}(a-b)b^{\pi}-ww^{\pi}\\
&&+ab^{\pi}-ba^{\pi}+aa^{\pi}-bb^{\pi}-2aa^{\pi}b^{\pi}+2bb^{\pi}a^{\pi}]\\
&=&(a-b)-((a-b)+bb^{\pi}a^{\pi}-aa^{\pi}b^{\pi}-ww^{\pi})\\
&=&bb^{\pi}a^{\pi}-aa^{\pi}b^{\pi}-ww^{\pi}.
\end{eqnarray*}
Note that $(bb^{\pi}a^{\pi}-aa^{\pi}b^{\pi})^{k}=(b-a)^{k}b^{\pi}a^{\pi}$ and $(b-a)^{k}=\sum_{i+j=k}\lambda_{i,j}b^{j}a^{i}$.
Let $k\geq 2\max\{s, t\}$. Then we have $(bb^{\pi}a^{\pi}-aa^{\pi}b^{\pi})^{k}=0$.

Since $(bb^{\pi}a^{\pi}-aa^{\pi}b^{\pi})ww^{\pi}=ww^{\pi}(bb^{\pi}a^{\pi}-aa^{\pi}b^{\pi})=0$, we have $ bb^{\pi}a^{\pi}-aa^{\pi}b^{\pi}-ww^{\pi}$ is nilpotent.

Hence, we get $(a-b)^{D}=w^{D}+a^{D}(1-bb^{\pi}a^{D})^{-1}b^{\pi}-a^{\pi}(1-b^{D}aa^{\pi})^{-1}b^{D}.$

For the `` only if '' part: Assume $(a-b)\in \mathcal{A}^{D}.$ Since
$(bb^{D})^2 = bb^{D}$, $bb^{D}\in \mathcal{A}^{D}.$ By Lemma 2.2 and
$(a-b)bb^{D}=bb^{D}(a-b)$, then we have $(a-b)bb^{D}\in
\mathcal{A}^{D}.$ Similarly, since
$aa^{D}(a-b)bb^{D}=(a-b)bb^{D}aa^{D}$, we have $aa^{D}(a-b)bb^{D}\in
\mathcal{A}^{D}$.
\end{proof}

\section { \bf Under the condition $a^{3}b=ba$ , $b^{3}a=ab$.}

In \cite{L1}, Liu et al. give the explicit representations of $(a+b)^{D}$ of two complex matrices under the condition $a^{3}b=ba$ and $b^{3}a=ab$. In this section, we will extend the result to a ring $R$ in which 2 is an unit.

\begin{lemma}
Let $a,b\in R$ be such that $a^{3}b=ba$ and $b^{3}a=ab$. Then

$(1)$ $ba^{i}=a^{3i}b$ and $b^{i}a=a^{3^{i}}b^{i}$.

$(2)$ $ab^{i}=b^{3i}a$ and $a^{i}b=b^{3^{i}}a^{i}$.

$(3)$ $ab=a^{26i}(ab)b^{2i}$ and $ba=b^{26i}(ba)a^{2i}$.
\end{lemma}
\begin{proof}
(1) It is easy to get $ba^{i}=baa^{i-1}=a^{3}ba^{i-1}=\cdots=a^{3i}b$.
Similarly, we have $b^{i}a=b^{i-1}ba=b^{i-1}a^{3}b=b^{i-2}ba^{3}b=b^{i-2}a^{3^{2}}b^{2}=\cdots=a^{3^{i}}b^{i}$.

(2) By (1), it is easy to get $ab^{i}=b^{3i}a$\ and $a^{i}b=b^{3^{i}}a^{i}$.

(3) $ab=b^{3}a=a^{27}b^{3}=a^{26}(ab)b^{2}=a^{26i}(ab)b^{2i}$.

Similarly, we get $ba=b^{26i}(ba)a^{2i}$.
\end{proof}
\begin{lemma}
Let $a,b\in R^{D}$ be such that $a^{3}b=ba$ and $b^{3}a=ab$. Then

$(1)$ $(a^{D})^{3}b=ba^{D}$ and $(b^{D})^{3}a=ab^{D}$.

$(2)$ $aa^{D}$ commutes with $b$ and $b^{D}$.

$(3)$ $bb^{D}$ commutes with $a$ and $a^{D}$.

$(4)$ $ab^{D}=b^{D}a^{3}$ and $a^{D}b=a^{D}b^{3}$.

$(5)$ $a^{D}b^{D}=b^{D}(a^{D})^{3}$ and $b^{D}a^{D}=a^{D}(b^{D})^{3}$.

$(6)$ $a^{D}b^{D}=b^{D}a^{D}b^{2}$ and $b^{D}a^{D}=a^{D}b^{D}a^{2}$.
\end{lemma}
\begin{proof}
(1) Let $k$ = ind$(a)$. Then
\begin{eqnarray*}
ba^{D}&=&ba^{k}(a^{D})^{k+1}=(a)^{3k}b(a^{D})^{k+1}=(a^{D})^{3}a^{3(k+1)}b(a^{D})^{k+1}=(a^{D})^{3}ba^{k+1}(a^{D})^{k+1}\\
&=&(a^{D})^{3}ba^{k}(a^{D})^{k}=(a^{D})^{3}a^{3k}b(a^{D})^{k}=(a^{D})^{3\times2}a^{3(k+1)}b(a^{D})^{k}=(a^{D})^{3\times2}ba^{k+1}(a^{D})^{k}\\
&=&(a^{D})^{3\times2}ba^{k}(a^{D})^{k-1}=\cdots=(a^{D})^{3\times(k+1)}ba^{k}=(a^{D})^{3\times(k+1)}a^{3k}b=(a^{D})^{3}b.
\end{eqnarray*}
Similarly, we have $(b^{D})^{3}a=ab^{D}$.

(2) By hypothesis and (1), we get $baa^{D}=a^{3}ba^{D}=a^{3}(a^{D})^{3}b=aa^{D}b$. Then $b^{D}aa^{D}=aa^{D}b^{D}$.
Similarly, we also have (3) is hold.

(4) By (3), we get $b^{D}a^{3}=b^{D}a^{3}bb^{D}=b^{D}bab^{D}=ab^{D}bb^{D}=ab^{D}$. Similarly, $a^{D}b=a^{D}b^{3}$.

(5) By (1) and (3), we have $b^{D}(a^{D})^{3}=b^{D}(a^{D})^{3}bb^{D}=b^{D}ba^{D}b^{D}=a^{D}b^{D}bb^{D}=a^{D}b^{D}$. Similarly, $a^{D}(b^{D})^{3}=b^{D}a^{D}.$

(6) By (5), we have $b^{D}a^{D}b^{2}=a^{D}(b^{D})^{3}b^{2}=a^{D}b^{D}.$ Similarly, $b^{D}a^{D}=a^{D}b^{D}a^{2}$.
\end{proof}
\begin{lemma}
Let $a,b\in R^{D}$ be such that $ab^{3}=ba$ and $ba^{3}=ab$. Then $a^{D}b^{D}=b^{3}a$ and $b^{D}a^{D}=a^{3}b$.
\end{lemma}
\begin{proof}
Let $k$ = ind$(a)$. Then
\begin{eqnarray*}
a^{D}b&=&a^{k}(a^{D})^{k+1}b=(a^{D})^{k+1}ba^{3k}=(a^{D})^{k+1}ba^{3(k+1)}(a^{D})^{3}=(a^{D})^{k+1}a^{k+1}b(a^{D})^{3}\\
&=&(a^{D})^{k}a^{k}b(a^{D})^{3}=(a^{D})^{k}ba^{3k}(a^{D})^{3}=(a^{D})^{k}ba^{3(k+1)}(a^{D})^{3\times2}=(a^{D})^{k}a^{k+1}b(a^{D})^{3\times2}\\
&=&(a^{D})^{k-1}a^{k}b(a^{D})^{3\times2}=\cdots=a^{k}b(a^{D})^{3\times(k+1)}=ba^{3k}(a^{D})^{3\times(k+1)}=b(a^{D})^{3}.
\end{eqnarray*}
Similarly, $b^{D}a=a(b^{D})^{3}.$\\
Then we have
$aa^{D}b=ab(a^{D})^{3}=ba^{3}(a^{D})^{3}=ba^{D}a.$ Similarly, we get $bb^{D}a=ab^{D}b.$\\
Hence, we can obtain that
\begin{eqnarray*}
a^{3}b^{D}&=&a^{3}b(b^{D})^{2}=ba^{9}(b^{D})^{2}\\
&=&b^{D}b^{2}a^{9}(b^{D})^{2}=b^{D}ba^{3}b(b^{D})^{2}\\
&=&b^{D}abb(b^{D})^{2}=b^{D}a(bb^{D})^{2}=b^{D}a.
\end{eqnarray*}
So, we get
$a^{D}b^{D}=a^{D}b^{D}aa^{D}=a^{D}a(b^{D})^{3}a^{D}=(b^{D})^{3}a^{D}$ and
$b^{D}a^{D}=(a^{D})^{3}b^{D}.$

Similar to the proof of Lemma 3.1, we have $ab=a^{2i}(ab)b^{26i}$. Then it is easy to get
\begin{eqnarray*}
&&abb^{D}a^{D}=bb^{D}aa^{D}=b^{D}ba^{D}a=b^{D}a^{D}ab,\\
&&b^{D}a^{D}abb^{D}a^{D}=b^{D}bb^{D}a^{D}aa^{D},\\
&&(ab)^{2}b^{D}a^{D}=(ab)abb^{D}a^{D}=(ab)aa^{D}bb^{D}=a^{2i}(ab)b^{26i}aa^{D}bb^{D}=a^{2i}(ab)b^{26i}=ab.
\end{eqnarray*}
Then this implies that
$(ab)^{\sharp}=b^{D}a^{D}$ and $(ba)^{\sharp}=a^{D}b^{D}$.

Hence, there exist $i\in \mathbb{N}$ such that
\begin{eqnarray*}
b^{D}a^{D}&=&(ab)^{\sharp}=((ab)^{\sharp})^{2}ab=b^{D}a^{D}b^{D}a^{D}ba^{3}=b^{D}a^{D}b^{D}(a^{D}a)b^{3}a^{2}=b^{D}a^{D}b^{D}b^{3}a^{2}\\
&=&b^{D}a^{D}(b^{D}b)b^{2}a^{2}=b^{D}a^{D}b^{2}a^{2}=b^{D}a^{D}bab^{3}a=b^{D}a^{D}ab^{6}a=b^{D}b^{6}a^{2}a^{D}\\
&=&b^{4}(b^{D}b)(ba)(a^{D}a)=b^{4}(b^{D}b)b^{2i}(ba)a^{26i}(a^{D}a)=b^{4}b^{2i}(ba)a^{26i}=b^{4}ba\\
&=&b^{5}a.
\end{eqnarray*}
and $a^{D}b^{D}=(ab)^{\sharp}=a^{5}b.$

So, we have
$a^{D}b^{D}=(b^{D})^{3}a^{D}=(b^{D})^{2}b^{D}a^{D}=(b^{D})^{2}b^{5}a=b^{2}(bb^{D})ba=b^{2}(bb^{D})b^{2i}baa^{26i}=b^{3}a$ and
$b^{D}a^{D}=a^{3}b.$
\end{proof}
\begin{lemma}
Let $a,b\in R^{D}$ be such that $a^{3}b=ba$ and $b^{3}a=ab$. Then the following statements hold:

$(1)$ $a^{D}b^{D}=(b^{D})^{3}a^{D}=b^{D}a^{D}a^{2}=b^{2}b^{D}a^{D}$.

$(2)$ $b^{D}a^{D}=(a^{D})^{3}b^{D}=a^{D}b^{D}b^{2}=a^{2}a^{D}b^{D}$.
\end{lemma}
\begin{proof}
Let $a^{D}=x\in R^{\sharp}$ and $b^{D}=y\in R^{\sharp}$. By Lemma 3.2, we have $xy^{3}=yx$ and $yx^{3}=xy$.
Then by Lemma 3.3, it follows $x^{\sharp}y^{\sharp}=y^{3}x$ and $y^{\sharp}x^{\sharp}=x^{3}y$, that is, $a^{2}a^{D}b^{2}b^{D}=(b^{D})^{3}a^{D}$.

Note that $$a^{2}a^{D}b^{2}b^{D}=a^{2}a^{D}b^{3}(b^{D})^{2}=a^{2}ba^{D}(b^{D})^{2}=a^{2}(a^{D})^{3}b(b^{D})^{2}=a^{D}b^{D},$$
and
$$b^{D}a^{D}a^{2}=b^{D}a^{3}(a^{D})^{2}=ab^{D}(a^{D})^{2}=aa^{D}(b^{D})^{3}a^{D}=(b^{D})^{3}a^{D}aa^{D}=(b^{D})^{3}a^{D}.$$
So, we get $a^{D}b^{D}=(b^{D})^{3}a^{D}=b^{D}a^{D}a^{2}$. Similarly, $b^{D}a^{D}=(a^{D})^{3}b^{D}=a^{D}b^{D}b^{2}.$\\
Hence, by $b^{D}a^{D}=a^{D}b^{D}b^{2}$ and Lemma 3.2 (6), we have
\begin{eqnarray*}
b^{2}b^{D}a^{D}&=&b^{D}b(ba^{D})=b^{D}ba^{D}b^{3}=a^{D}b^{D}b^{4}=b^{D}a^{D}b^{2}=a^{D}b^{D}.
\end{eqnarray*}
Similarly, $a^{2}a^{D}b^{D}=b^{D}a^{D}$.
\end{proof}

\begin{lemma}
Let $a,b\in R^{D}$ be such that $a^{3}b=ba$ and $b^{3}a=ab$. Then the following statements hold:

$(1)$ $aa^{D}a^{4+i}b^{j}bb^{D}=aa^{D}a^{i}b^{j}bb^{D}$.

$(2)$ $aa^{D}a^{2+i}b^{2+j}bb^{D}=aa^{D}a^{i}b^{j}bb^{D}$, where $i,j\in \mathbb{N}$.

$(3)$ $aa^{D}abb^{D}=a^{D}(b^{D})^{2}$.

$(4)$ $aa^{D}a^{3}bb^{D}=a^{D}bb^{D}$.

$(5)$ $aa^{D}a^{2}bbb^{D}=aa^{D}b^{D}$.

$(6)$ $aa^{D}ab^{2}bb^{D}=a^{D}bb^{D}$.

$(7)$ $ab(1-aa^{D})=0.$

$(8)$ $ba(1-bb^{D})=0.$
\end{lemma}
\begin{proof}
(1) By Lemma 3.4 (2), we have $$aa^{D}a^{4}bb^{D}=aa^{D}abab^{D}=aba^{2}a^{D}b^{D}=abb^{D}a^{D}=aa^{D}bb^{D}.$$
Then we get
\begin{eqnarray*}
aa^{D}a^{4+i}b^{j}bb^{D}&=&aa^{D}a^{4}bb^{D}aa^{D}a^{i}b^{j}bb^{D}\\
&=&aa^{D}bb^{D}aa^{D}a^{i}b^{j}bb^{D}\\
&=&aa^{D}a^{i}b^{j}bb^{D}.
\end{eqnarray*}

(2) Note that $a^{2}a^{D}b^{2}b^{D}=a^{D}b^{D},$
Then we have $aa^{D}a^{2}b^{2}bb^{D}=a(a^{2}a^{D}b^{2}b^{D})b=aa^{D}bb^{D}$.
This implies that
\begin{eqnarray*}
aa^{D}a^{2+i}b^{2+j}bb^{D}&=&aa^{D}a^{i}aa^{D}a^{2}b^{2}bb^{D}b^{j}bb^{D}\\
&=&aa^{D}a^{i}aa^{D}bb^{D}b^{j}bb^{D}\\
&=&aa^{D}a^{i}b^{j}bb^{D}.
\end{eqnarray*}
Similarly, we can get $aa^{D}b^{2+i}a^{2+j}bb^{D}=aa^{D}a^{i}b^{j}bb^{D}$.

(3) By Lemma 3.4 (2), we have $aa^{D}abb^{D}=a^{2}a^{D}b^{D}b=b^{D}a^{D}b=a^{D}(b^{D})^{3}b=a^{D}(b^{D})^{2}$.

(4) $aa^{D}a^{3}bb^{D}=aa^{D}bab^{D}=ba^{2}a^{D}b^{D}=bb^{D}a^{D}=a^{D}bb^{D}.$

(5) In the proof of Lemma 3.4 (1), we get $a^{2}a^{D}b^{2}b^{D}=a^{D}b^{D}$.
Then we have $$aa^{D}a^{2}bb^{D}=a(aa^{D}abbb^{D})=aa^{D}b^{D}.$$

(6) Similar to (5), we have $aa^{D}ab^{2}bb^{D}=(aa^{D}abbb^{D})b=a^{D}bb^{D}.$

(7) $ab(1-aa^{D})=(1-aa^{D})ab=(1-aa^{D})a^{26k}(ab)b^{2k}=0.$

(8) $ba(1-bb^{D})=(1-bb^{D})ba=(1-bb^{D})b^{26k}(ba)a^{2k}=0.$
\end{proof}
\begin{theorem}
Let $a,b\in R^{D}$ be such that $a^{3}b=ba$ and $b^{3}a=ab$. Suppose 2 is an unit of $R$. Then
$a+b$ is Drazin invertible and

$(a+b)^{D}=\frac{1}{8}bb^{D}(3a^{3}+3b^{3}-a-b)aa^{D}+a^{D}(1-bb^{D})+(1-aa^{D})b^{D}$.
\end{theorem}
\begin{proof}
First, let $M=M_{1}+M_{2}+M_{3}$, where $M_{1}=\frac{1}{8}bb^{D}(3a^{3}+3b^{3}-a-b)aa^{D}$, $M_{2}=a^{D}(1-bb^{D})$, $M_{3}=(1-aa^{D})b^{D}$.
In what follows, we show that $M$ is the
Drazin inverse of $a+b$, i.e. the following conditions hold: (a) $M(a+b)=(a+b)M$,
(b) $M(a+b)M=M$ and (c) $(a+b)-(a+b)^{2}M$ is nilpotent.\\
(a) By Lemma 3.2 (2) and (3), we have
$$(a+b)M_{1}=\frac{1}{8}bb^{D}(a+b)(3a^{3}+3b^{3}-a-b)aa^{D}$$
and $$M_{1}(a+b)=\frac{1}{8}bb^{D}(3a^{3}+3b^{3}-a-b)(a+b)aa^{D}.$$
After a calculation we obtain $(a+b)bb^{D}(3a^{3}+3b^{3}-a-b)aa^{D}=bb^{D}(3a^{3}+3b^{3}-a-b)aa^{D}(a+b)$, this implies
$M_{1}(a+b)=(a+b)M_{1}$.\\
By Lemma 3.2 (2) and (3), we get
\begin{eqnarray}
\nonumber
(a+b)M_{2}-M_{2}(a+b)&=&(aa^{D}+ba^{D})(1-bb^{D})-(a^{D}a+a^{D}b)(1-bb^{D})\\ \nonumber
                     &=&(ba^{D}-a^{D}b)(1-bb^{D})\\ \nonumber
                     &=&((a^{D})^{3}b-a^{D}b)(1-bb^{D})\\ \nonumber
                     &=&((a^{D})^{4}-(a^{D})^{2})ab(1-bb^{D})\\
                     &=&0.
\end{eqnarray}
Note that $ba(1-aa^{D})=0$, then we have
\begin{eqnarray}
\nonumber
(a+b)M_{3}-M_{3}(a+b)&=&(1-aa^{D})(ab^{D}+bb^{D})-(1-aa^{D})(b^{D}a+b^{D}b)\\ \nonumber
                     &=&(1-aa^{D})(ab^{D}-b^{D}a)\\ \nonumber
                     &=&((b^{D})^{3}a-b^{D}a)(1-aa^{D})\\ \nonumber
                     &=&((b^{D})^{4}-(b^{D})^{2})ba(1-aa^{D})\\
                     &=&0.
\end{eqnarray}
By (2.3) and (2.4), we can obtain  $(a+b)M=M(a+b)$.

(b) By Lemma 3.2 (3), we get
\begin{eqnarray*}
M_{1}(a+b)M_{2}&=&\frac{1}{8}bb^{D}(3a^{3}+3b^{3}-a-b)aa^{D}(a+b)a^{D}(1-bb^{D})\\
               &=&\frac{1}{8}(3a^{3}+3b^{3}-a-b)aa^{D}(a+b)a^{D}(1-bb^{D})bb^{D}\\
               &=&0.
\end{eqnarray*} Similarly, we have
\begin{eqnarray*}
M_{1}(a+b)M_{3}&=&\frac{1}{8}bb^{D}(3a^{3}+3b^{3}-a-b)aa^{D}(a+b)(1-aa^{D})b^{D}\\
               &=&\frac{1}{8}bb^{D}(3a^{3}+3b^{3}-a-b)(a+b)aa^{D}(1-aa^{D})b^{D}\\
               &=&0
\end{eqnarray*}
and $$M_{2}(a+b)M_{1}=M_{2}(a+b)M_{3}=M_{3}(a+b)M_{1}=M_{3}(a+b)M_{2}=0.$$
By hypothesis and Lemma 3.5, we can simplify
\begin{eqnarray*}
M_{1}(a+b)M&=&\frac{1}{8}bb^{D}(3a^{3}+3b^{3}-a-b)aa^{D}(a+b)\frac{1}{8}bb^{D}(3a^{3}+3b^{3}-a-b)aa^{D}\\
           &=&\frac{1}{64}bb^{D}(3a^{3}+3b^{3}-a-b)(a+b)(3a^{3}+3b^{3}-a-b)aa^{D}\\
           &=&\frac{1}{64}bb^{D}(9a^{7}+9a^{4}b^{3}+9a^{3}ba^{3}+9a^{3}b^{4}+9b^{3}a^{4}+9b^{3}ab^{3}+9b^{4}a^{3}+9b^{7}\\
           &&-3a^{5}-3a^{2}b^{3}-3aba^{3}-3ab^{4}-3ba^{4}-3bab^{3}-3b^{2}a^{3}-3b^{5}\\
           &&-3a^{5}-3a^{4}b-3a^{3}ba-3a^{3}b^{2}-3b^{3}a^{2}-3b^{3}ab-3b^{4}a-3b^{5}\\
           &&+a^{3}+a^{2}b+aba+ab^{2}+ba^{2}+bab+b^{2}a+b^{3})aa^{D}\\
           &=&\frac{1}{64}bb^{D}(9a^{7}+9a^{4}b^{3}+9ba^{4}+9a^{3}b^{4}+9b^{3}a^{4}+9ab^{4}+9b^{4}a^{3}+9b^{7}\\
           &&-3a^{5}-3a^{2}b^{3}-3b^{3}a^{4}-3ab^{4}-3ba^{4}-3a^{3}b^{4}-3b^{2}a^{3}-3b^{5}\\
           &&-3a^{5}-3a^{4}b-3a^{6}b-3a^{3}b^{2}-3b^{3}a^{2}-3ab^{2}-3b^{4}a-3b^{5}\\
           &&+a^{3}+a^{2}b+a^{4}b+ab^{2}+a^{6}b+a^{3}b^{2}+a^{9}b^{2}+b^{3})aa^{D}\\
           &=&\frac{1}{64}bb^{D}(9a^{3}+9b^{3}+9b+9a^{3}+9b^{3}+9a+9a^{3}+9b^{3}\\
           &&-3a-3b-3b^{3}-3a-3b-3a^{3}-3a-3b\\
           &&-3a-3b-3a^{2}b-3a-3b-3ab^{2}-3a-3b\\
           &&+a^{3}+a^{2}b+b+ab^{2}+a^{2}b+a+ab^{2}+b^{3})aa^{D}\\
           &=&\frac{1}{64}bb^{D}(25a^{3}+25b^{3}-ab^{2}-a^{2}b-8a-8b)aa^{D}\\
           &=&\frac{1}{64}(25a^{D}bb^{D}+25b^{D}aa^{D}-a^{D}bb^{D}-b^{D}aa^{D}-8a^{D}b^{D}b^{D}-8b^{D}a^{D}a^{D})\\
           &=&\frac{1}{64}(24a^{D}bb^{D}+24b^{D}aa^{D}-8a^{D}b^{D}b^{D}-8b^{D}a^{D}a^{D})\\
           &=&\frac{1}{8}bb^{D}(3a^{3}+3b^{3}-a-b)aa^{D}.
\end{eqnarray*}
Note that $a^{D}ba^{D}(1-bb^{D})=a^{D}(1-bb^{D})baa^{D}a^{D}=0$ and $b^{D}ab^{D}(1-aa^{D})=0$.
After a calculation, we obtain
\begin{eqnarray*}
M(a+b)M&=&M_{1}+M_{2}(a+b)M_{2}+M_{3}(a+b)M_{3}\\
&=&M_{1}+a^{D}(a+b)a^{D}(1-bb^{D})+b^{D}(a+b)b^{D}(1-aa^{D})\\
&=&M_{1}+(a^{D}+a^{D}ba^{D})(1-bb^{D})+(b^{D}ab^{D}+b^{D})(1-aa^{D})\\
&=&M.
\end{eqnarray*}

(c) Note that $(aba^{D}+baa^{D}+bba^{D})(1-bb^{D})=0$ and $(aab^{D}+abb^{D}+bab^{D})(1-aa^{D})=0$.

Similar to the proof of (b), by Lemma 3.5, we have
\begin{eqnarray*}
(a+b)^{2}M&=&(a+b)[\frac{1}{8}bb^{D}(3a^{4}+3ab^{3}+3ba^{3}+3b^{4}-a^{2}-ab-ba-b^{2})aa^{D}\\
&&+(aa^{D}+ba^{D})(1-bb^{D})+(1-aa^{D})(ab^{D}+bb^{D})]\\
&=&\frac{1}{8}bb^{D}(3a^{5}+3a^{2}b^{3}+3aba^{3}+3ab^{4}+3ba^{4}+3bab^{3}+3a^{2}b^{3}+3b^{5}\\
&&-a^{3}-a^{2}b-aba-ab^{2}-ba^{2}-bab-b^{2}a-b^{3})aa^{D}+(a^{2}a^{D}+aba^{D}\\
&&+baa^{D}+bba^{D})(1-bb^{D})+(aab^{D}+abb^{D}+bab^{D}+bbb^{D})(1-aa^{D})\\
&=&\frac{1}{8}(8a^{D}b^{D}b^{D}+8b^{D}a^{D}a^{D})+a^{2}a^{D}(1-bb^{D})+b^{2}b^{D}(1-aa^{D})\\
&=&\frac{1}{8}(8a^{D}b^{D}b^{D}+8b^{D}a^{D}a^{D})+a^{2}a^{D}-a^{D}b^{D}b^{D}+b^{2}b^{D}-b^{D}a^{D}a^{D}\\
&=&a^{2}a^{D}+b^{2}b^{D}.
\end{eqnarray*}
Then $a+b-(a+b)^{2}M=aa^{\pi}+bb^{\pi}$, since $aa^{\pi}bb^{\pi}=bb^{\pi}aa^{\pi}=0$.
Hence, this shows $a+b-(a+b)^{2}M$ is nilpotent.
\end{proof}

\centerline {\bf ACKNOWLEDGMENTS} This research is supported by the
National Natural Science Foundation of China (10971024), the
Specialized Research Fund for the Doctoral Program of Higher
Education (20120092110020), the Natural Science Foundation of
Jiangsu Province(BK2010393) and the Foundation of Graduate
Innovation Program of Jiangsu Province(CXZZ12-0082).


\begin{thebibliography}{s2}

\bibitem{Cas1} N. Castro-Gonz\'{a}lez, Additive perturbation results for the Drazin inverse, Linear Algebra Appl. 397 (2005), 279-297.

\bibitem{Cas2} N. Castro-Gonz\'{a}lez and J.J. Koliha, New additive results for the g-Drazin inverse, Proc. Roy. Soc. Edinburgh Sect. A 134 (2004), 1085-1097.

\bibitem{Cas3} N. Castro-Gonz\'{a}lez and M.F. Martinez-Serrano, Expressions for the g-Drazin inverse of additive perturbed elements in a Banach algebra, Linear Algebra Appl. 432 (2010), 1885-1895.

\bibitem{C} J.L. Chen, G.F. Zhuang and Y.M. Wei, The Drazin inverse of a sum of morphisms, Acta. Math. Sci. 29A (2009), 538-552 (in Chinese).

\bibitem{DS1} D.S. Cvetkovi\'{c}-Ili\'{c}, The generalized Drazin inverse with commutativity up to a factor in a Banach algebra, Linear Algebra Appl. 431 (2009), 783-791.

\bibitem{DS2} D.S. Cvetkovi\'{c}-Ili\'{c}, D.S. Djordjevi\'{c} and Y.M. Wei, Additive results for the generalized Drazin inverse in a Banach algebra, Linear Algebra Appl. 418 (2006), 53-61.

\bibitem{DS3} D.S. Cvetkovi\'{c}-Ili\'{c}, X.J. Liu and Y.M. Wei, Some additive results for the generalized Drazin inverse in Banach algebra, Elect. J. Linear Algebra, 22 (2011), 1049-1058.

\bibitem{D1}  C.Y. Deng, The Drazin inverse of bounded operators with commutativity up to a factor, Appl. Math. Comput. 206 (2008), 695-703.

\bibitem{D2}  C.Y. Deng and Y.M. Wei, New additive results for the generalized Drazin inverse, J. Math. Anal. Appl. 370 (2010), 313-321.

\bibitem{Dj1} D.S. Djordjevi\'{c} and Y.M. Wei, Additive results for the generalized Drazin inverse, J. Austral. Math. Soc. 73 (2002), 115-125.

\bibitem{Dj2} D.S. Djordjevi\'{c} and Y.M. Wei, Outer generalized inverses in rings, Comm. Algebra 33 (2005), 3051-3060.

\bibitem{DRA}  M.P. Drazin, Pseudo-inverses in associative rings and semigroups, Amer. Math. Monthly 65 (1958), 506-514.

\bibitem{HARTE} R.E. Harte, Spectral projections, Irish Math. Soc. Newslett. 11 (1984), 10-15.

\bibitem{HARTWING1} R.E. Hartwig and J.M. Shoaf, Group inverses and Drazin inverses of bidiagonal and triangular Toeplitz matrices, J. Austral. Math. Soc. 24 (1977), 10-34.

\bibitem{HARTWING2} R.E. Hartwig, G.R. Wang and Y.M. Wei, Some additive results on Drazin inverse, Linear Algebra Appl. 322 (2001), 207-217.

\bibitem{K1} J.J. Koliha, A generalized Drazin inverse, Glasgow Math. J. 38 (1996), 367-381.

\bibitem{K2} J.J. Koliha and P. Patr\'{\i}cio, Elements of rings with equal spectral idempotents, J. Austral. Math. Soc. 72 (2002), 137-152.

\bibitem{L1}  X.J. Liu, L. Xu and Y.M. Yu, The representations of the Drazin inverse of differences of two matrices, Appl. Math. Comput.216 (2010) 3652-3661.

\bibitem{P} P. Patr\'{\i}cio and R.E. Hartwig, Some additive results on Drazin inverse, Appl. Math. Comput. 215 (2009), 530-538.

\bibitem{P2} R. Puystjens and M.C. Gouveia, Drazin invertibility for matrices over an arbitrary ring, Linear Algebra Appl. 385 (2004), 105-116.

\bibitem{P3}  R. Puystjens and R.E. Hartwig, The group inverse of a companion matrix, Linear Multilinear Algebra 43 (1997), 137-150.

\bibitem{W1} Y.M. Wei and C.Y. Deng, A note on additive results for the Drazin inverse, Linear Multilinear Algebra 59 (2011), 1319-1329.

\bibitem{W2} Y.M. Wei and G.R. Wang, The perturbation theory for the Drazin inverse and its applications, Linear Algebra Appl. 258 (1997), 179-186.

\bibitem{Z} G.F. Zhuang, J.L. Chen, D.S. Cvetkovi$\acute{c}$-Ili$\acute{c}$ and Y.M. Wei, Additive property of Drazin invertibility of elements in a ring,  Linear Multilinear Algebra 60 (2012), 903-910.

\end{thebibliography}
\end{document}